\documentclass{amsart}

\usepackage{amssymb}
\usepackage{cite}

\newtheorem{theorem}{Theorem}[section]
\newtheorem{lemma}[theorem]{Lemma}
\newtheorem{proposition}[theorem]{Proposition}
\newtheorem{corollary}[theorem]{Corollary}
\theoremstyle{definition}

\theoremstyle{remark}
\newtheorem{remark}[theorem]{Remark}
\numberwithin{equation}{section}

\newcommand{\R}{{\mathbb R}}

\newcommand{\e}{\epsilon}
\newcommand{\C}{\mathcal{C}}
\newcommand{\M}{\mathcal{M}}

\author{Mark Allen}
\address{Department of Mathematics, Brigham Young University, Provo,
  UT 84602}
\email{allen@mathematics.byu.edu}


\title[Free Boundary Problem]{A Free boundary problem\\ on three-dimensional cones}

\begin{document}

\maketitle
 \begin{abstract}
  We consider a free boundary problem on cones depending on a parameter $c$ and study when the free boundary 
  is allowed to pass through the vertex of the cone. We show that when the cone is three-dimensional and $c$ is large enough, the free boundary avoids the vertex. 
  We also show that when 
  $c$ is small enough but still positive, the free boundary is allowed to pass through the vertex. This establishes $3$ as the critical dimension for which 
  the free boundary may pass through the vertex of a right circular cone. In view of the well-known connection between area-minimizing 
  surfaces and the free boundary problem under consideration, our result is analogous to a result of Morgan that classifies when an area-minimizing surface
  on a cone passes through the vertex.  
 \end{abstract}

\section{Introduction}
 We study solutions to the problem 
    \begin{equation}   \label{e:onephase}
     \begin{aligned} 
      \Delta u &=0  &\text{ in } &\{u>0\} \\
         |\nabla u|&=1  &\text{ on } &\partial \{u>0\}.
   \end{aligned}
   \end{equation}
 on right circular cones in $\R^n$. We are interested in determining when the free boundary $\partial \{u>0\}$ is allowed to pass through the vertex of the cone. 
 
 The above problem has applications to two dimensional flow problems as well as heat flow problems 
 (see \cite{ac81} where \eqref{e:onephase} was first studied). When considering the applications on a manifold, one studies a variable coefficient problem in divergence form: 
  \begin{equation}    \label{e:variphase}
    \begin{aligned} 
      \partial_j(a^{ij}(x)u_i) &=0  &\text{ in } &\{u>0\} \\
         a^{ij}(x)u_i u_j &=1  &\text{ on } &\partial \{u>0\}.
   \end{aligned}\
  \end{equation}
 Solutions of \eqref{e:variphase} may be found inside a bounded domain $\Omega$ by minimizing the functional:
  \begin{equation}   \label{e:functional2}
   \int_{\Omega} a^{ij}(x)v_iv_j + Q(x)\chi_{\{v>0\}}. 
  \end{equation}
  However, since the functional is not convex, minimizers of \eqref{e:functional2} may not be unique and there exist solutions to \eqref{e:variphase} 
  which are not minimizers of \eqref{e:functional2}. When the coefficients $a^{ij}(x)$ are Lipschitz continuous and satisfy an ellipticity condition, 
  regularity of the free boundary was studied in \cite{fs07}. The authors in \cite{fs07} adapted the sup-convolution
  approach of Caffarelli in \cite{c87,c88,c89} for viscosity solutions. This approach relies on a nondivergence structure and therefore requires Lipschitz continuity of the coefficients
  $a^{ij}(x)$ so that \eqref{e:functional2} can be transformed into a nondivergence operator. More recently, the regularity of the free boundary for H\"older continuous 
  coefficients $a^{ij}(x)$ was accomplished in \cite{dfs16} using different techniques. For coefficients $a^{ij}(x)$ assumed merely to be bounded, measurable, and satisfying the usual
  ellipticity conditions, regularity of the solution and its growth away from the free boundary was studied in \cite{pt16}. However, to date nothing is know regarding the regularity
  of the free boundary when the coefficients $a^{ij}(x)$ are allowed to be discontinuous. In this paper we are interested in how the free boundary interacts with 
  isolated discontinuous points of the coefficients $a^{ij}(x)$. 
  In the context of a hypersurface, these points are considered to be a topological singularity. The simplest such case is the vertex of a cone. The aim of
  this paper is to study when the free boundary of a solution that arises as a minimizer is allowed to pass through a topological singularity. Before stating the main 
  results of this paper we 
  first recall a connection between solutions to \eqref{e:onephase} and minimal surfaces in order to understand what results one might expect for the free boundary problem 
  on a cone. 
  
  \subsection{Connection to minimal surfaces}
    Results for the singular set of free boundary points are analogous to 
    results for the singular set of minimal surfaces. In the case of area-minimizing surfaces, the study of the singular set is reduced to considering  
    area-minimizing cones. Simons \cite{s68} showed that any area-minimizing cone in $\R^n$ for $n\leq 7$ is necessarily planar. Simons actually proved a stronger
    result in \cite{s68} by showing that any minimal stable cone is planar. He also provided an example of a cone in $\R^8$ that is stable and therefore
     a possible candidate for being an area-minimizing cone. One year later, it was shown that the Simons 
     cone is indeed area-minimizing, see \cite{bdg69}. As a consequence, $n=8$ is the first dimension for which a singularity of an area-minimizing hypersurface may occur. 
     
     Regarding the singular set of the free boundary for  minimizers, the authors in \cite{ac81} showed there are no singular points in dimension $n=2$. 
     In \cite{w99}
    a monotonicity formula is utilized to show that blow-up solutions are homogeneous, and therefore the free boundary of blow-up solutions is a cone. As a further consequence
     there exists a minimal dimension $k^*$ such that the singular set of the free boundary of minimizers is empty if the dimension $n<k^*$. 
     The authors in \cite{cjk04} showed $k^* >3$ and also provided an example of a nontrivial stable solution in dimension $n=7$. This example is analogous to the Simons
     cone and was later shown in \cite{dj09} to indeed be a minimizer. 
     Recently, the article \cite{js15} improved $k^*>4$. It is still an open problem as to whether $k^*=5,6,$ or $7$.

     The article \cite{t14} further strengthened the connection between minimal surfaces and the free boundary problem by establishing 
     a one-to-one correspondence between solutions of \eqref{e:onephase} in $\R^2$ and minimal bigraphs in $\R^3$. The one-to-one 
     correspondence further strengthens the principle of  
     a reduction in one dimension when moving from the theory of minimal surfaces to the one phase problem. Recall for instance that $k^* =8$ for area-minimizing surfaces 
     where as $k^*$ is most likely $7$ for minimizers of \eqref{e:functional2}. 
     
     \subsection{Area-minimizing surfaces and the free boundary problem on cones}
      In light of the connection described above, one may expect that results for the free boundary problem on cones are analogous 
      to the results for area-minimizing surfaces on cones. The description of the free boundary problem on cones \eqref{e:phase1} as well as the corresponding functional 
      \eqref{e:conefunct} is given later in Section \ref{s:prelim}. 
      
      On two-dimensional cones area-minimizing surfaces of co-dimension $1$ are distance minimizing geodesics. Two-dimensional cones
      in $\R^3$ are determined by the intersection of the cone with the two-sphere. If this intersection is a simple closed curve $\gamma$ on $S^2$, then the following Proposition 
      is well known.
       \begin{proposition}   \label{p:carmo}
        If length$(\gamma)<2\pi$, no distance minimizing geodesics pass through the vertex. If length $(\gamma)\geq2\pi$, then there are distance minimizing geodesics that 
        pass through the vertex. 
       \end{proposition}
      The first statement in Proposition \ref{p:carmo} can be found in Section 4-7 of \cite{d76}. The author with Chang Lara proved the following complete analogous result 
      \cite{acl15} for 
      minimizers of \eqref{e:conefunct} on a two-dimensional cone. 
      \begin{theorem}  \label{t:2cone}
        If $u$ is a minimizer of \eqref{e:conefunct} on a two-dimensional cone, and if length$(\gamma)<2\pi$, then the vertex of the cone $0\notin \partial \{u>0\}$. If
        length$(\gamma)\geq 2\pi$, then the free boundary can pass through the vertex. 
      \end{theorem}
      The proof of Theorem \ref{t:2cone} was the main result in \cite{acl15} and utilized that two-dimensional cones are isometrically flat. A competitor with a 
      smaller functional value was constructed via an iterative argument that depended on length$(\gamma)$. 
      
      In this paper we consider the free boundary problem on higher dimensional cones. 
      In view of the connection between area-minimizing surfaces and the one phase free boundary problem, 
      one is led to ask about area-minimizing surfaces on higher dimensional cones. Morgan \cite{m02} considered area-minimizing surfaces on 
      $n$-dimensional cones  defined by
       \begin{equation}   \label{e:cone}
        x_{n+1} = c \sqrt{\sum_{i=1}^n x_i^2}
       \end{equation}
       with $c \geq 0$, and proved that a $k$-dimensional plane through the vertex is area-minimizing if and only if 
        \begin{equation}  \label{e:relation}
         k \geq 3 \quad \text{ and } \quad \delta^2 \geq \frac{4(k-1)}{k^2},
        \end{equation}
      where $c=\delta^{-1}\sqrt{1-\delta^2}$. As a corollary \eqref{e:relation} also determines when a $k$-dimensional area minimizing hypersurface
      can pass through a cone for $3 \leq k \leq 7$. In particular,  a $3$-dimensional area minimizing hypersurface may pass through the vertex of a $4$-dimensional
      cone as given in \eqref{e:cone} and this is determined by \eqref{e:relation}. By the aforementioned drop in one dimension from area-minimizing surfaces to 
      the free boundary problem, one may expect that the lowest dimension for which a free boundary of a minimizer may pass through the vertex of a cone 
      of type \eqref{e:cone} is for a three dimensional cone, and this depends on the constant $c$. We prove this is indeed the case.

      \subsection{Main Results}
       We prove results analogous to those of area-minimizing surfaces on cones in \cite{m02}. In Section \ref{s:2ndvariation} we establish a second variational formula
       for minimizers of \eqref{e:conefunct} on a cone of type \eqref{e:cone}. With a notion of second variation one may discuss whether a solution is 
       stable. Our first result regards the stability of a homogeneous solution. 
              
       \begin{theorem}\label{t:main1}
       Let $\C$ be a three-dimensional cone of type \eqref{e:cone}. 
        There exists $0<c_0<\infty$ such that if $c \leq c_0$, then there exists a unique $($up to rotation$)$  $1$-homogeneous solution of \eqref{e:phase1} that is stable. 
        If $c>c_0$ then no $1$-homogeneous solution of \eqref{e:phase1} is stable.     
       \end{theorem}
       
       From the above theorem we obtain the following  
       \begin{corollary}  \label{c:avoidvertex}
          Let $\C$ be a three-dimensional cone of type \eqref{e:cone}, and let $c_0$ be the constant in Theorem \ref{t:main1}. If $c>c_0$ and $u$ is a minimizer 
          of \eqref{e:conefunct}, then the vertex $0 \notin \partial \{u>0\}$. 
       \end{corollary}

      In the history of area-minimizing surfaces and free boundary problems it is common for stable solutions to indeed be global minimizers. Furthermore, 
      the notion of stability and area-minimizing for hyperplanes on cones coincides \cite{m02}. Therefore, it is reasonable to assume that Theorem \ref{t:main1}
      may be improved by replacing the notion of stable with minimizer. In that vein we have our second main result. 
      
      \begin{theorem}  \label{t:main2}
       Let $\C$ be a three-dimensional cone of type \eqref{e:cone}. 
        There exists $c_1$ with $0<c_1 \leq c_0<\infty$ such that if $c \leq c_1$, then there exists a minimizer $u$ of \eqref{e:conefunct} such that 
         $0 \in \partial \{u>0\}$.        
      \end{theorem}
      
      The significance of Theorem \ref{t:main1} is that $c_1>0$ which shows that three is the lowest dimension for which the free boundary of a minimizer 
      passes through the vertex of a non-flat right circular cone. We expect that $c_1=c_0$ in Theorem \ref{t:main2}, but a proof that $c_1=c_0$ may 
      need to rely on numerical analysis such as in \cite{dj09}. 
      
      Many of the results in this paper apply to higher-dimensional cones of type \eqref{e:cone}. 
      In Section \ref{s:symmetric} we present a symmetric $1$-homogeneous solution $\Phi_c$ to \eqref{e:phase1} on $\C$ where the vertex $0 \in \partial\{u>0\}$. 
      This symmetric solution $\Phi_c$ has a variant in each dimension. 
      In Section \ref{s:stable} we show that if $v$ is any homogeneous stable solution with $0 \in \partial\{v>0\}$,  and $\C$ is three-dimensional, 
      then $v \equiv \Phi_c$. 
       Theorems \ref{t:main1} and \ref{t:main2} are reduced to showing whether the specific solution $\Phi_c$ is stable or is a minimizer. The symmetric candidate $\Phi_c$ is analogous
       to hyperplanes on cones as in  \cite{m02}. 
      As previously mentioned, there exist non-hyperplane cones that are area-minimizing; consequently,  the results in \cite{m02} for hyperplanes can only be used 
      to classify when an area-minimizing hypersurface of dimension $k$ passes through the vertex of a cone when $k \leq 7$. Similarly, in order to prove Theorems \ref{t:main1} and
      \ref{t:main2} on higher-dimensional cones, one would have to show any stable homogeneous solution $v\equiv \Phi_c$. Although we only present results in this paper
      for the symmetric solution $\Phi_c$ on a three-dimensional cone, the same techniques will apply in higher dimensions. We state this in the following remark which is 
      a partial analogue to Theorem 1.1 in \cite{m02} regarding hyperplanes on cones. 
      
      \begin{remark}
       Let $\C$ be a cone of type \eqref{e:cone}. Let $\Phi_c$ be the symmetric solution as defined in Section \ref{s:symmetric}. 
       If $n \geq 3$, then there exists two constants $0<c_1\leq c_0 < \infty$
       depending on dimension $n$ 
       such that $\Phi_c$ is stable if and only if $c\leq c_0$ and $\Phi_c$ is a minimizer of \eqref{e:conefunct} if $c \leq c_1$. 
      \end{remark}

      Finally, one may consider Lipschitz manifolds with isolated singularities. Suppose $\M$ is a three-dimensional manifold 
      and $\M$ is smooth in a neighborhood $\Omega \setminus \{x_0\}$.
      If there exists a sequence $r_k \to 0$ such that a rescaling $\M_k \to \C$ where $\C$ is of type \eqref{e:cone}, then the results of Theorem
      \ref{t:main2} will apply in a neighborhood of $x_0$. This is because if the free boundary of a minimizer  $u$ passes through $x_0$, 
      then using the regularity results in \cite{pt16}, one may obtain via a blow-up procedure a minimizer $u_0$ on $\C$ and the vertex $0 \in \partial\{u_0 >0\}$.

      \subsection{Outline and Notation} 
          The outline of this paper is as follows. In Section \ref{s:prelim} we define the notion of a solution to the free boundary problem and the corresponding functional. 
          We also state some preliminary results necessary later in the paper. In Section \ref{s:2ndvariation} we give a second variation formula for $1$-homogeneous
          solutions. In Section \ref{s:symmetric} we present a symmetric $1$-homogeneous solution $\Phi_c$. In Section \ref{s:stable} we classify when $\Phi_c$ is 
          a stable solution and prove Theorem \ref{t:main1} and Corollary \ref{c:avoidvertex}. In Section \ref{s:min} we show that for $c$ small but still positive the 
          solution $\Phi_c$ is a minimizer of the functional \eqref{e:conefunct} and thus prove Theorem \ref{t:main2}.

          We will use the following notation throughout the paper. 
          \begin{itemize}
           \item $n$ refers to dimension. 
           \item $\C$ is always a cone of type \eqref{e:cone}. 
           \item $c$ is always the constant appearing in the definition of a cone in \eqref{e:cone}. 
           \item $\nabla_c$ refers to the gradient on $\C$ arising from the inherited metric as explained in Section \ref{s:prelim}. 
           \item $\Delta_c$ refers to the Laplace-Beltrami operator on $\C$ as explained in \ref{s:prelim}. 
           \item $\nabla_{\theta}$ refers to the gradient on the sphere. 
           \item $\Delta_{\theta}$ represents the Laplace-Beltrami on the sphere. 
           \item $\Gamma := \{u>0\}$ where $u$ is $1$-homogeneous solution to \eqref{e:phase1}. 
          \end{itemize}

\section{Preliminaries}  \label{s:prelim}
 We consider a $3$-dimensional cone $\C$ in $\R^4$ given by 
  \begin{equation} \label{e:cone3}
  \C:= \{(y_1,y_2,y_3,y_4) \in \R^4 : y_4 = c\sqrt{y_1^2 + y_2^2 + y_3^2} \}.
  \end{equation} 
  We study minimizers of the functional 
    \begin{equation}  \label{e:conefunct}
     J(v,\Omega):= \int_{\Omega} |\nabla_c v|^2 + \chi_{\{v>0\}},
     \end{equation}
  where $\nabla_c v$ is the gradient on the cone $\C$ away from the vertex. 
   As shown later in Proposition \ref{p:solution}, minimizers of \eqref{e:conefunct} over $\Omega \subseteq \C$ are solutions to     
   \begin{equation}   \label{e:phase1}
     \begin{cases}
      u \geq 0 \\
      u\text{ is continuous in  } \Omega \\
           \Delta_c u =0  \text{ in } \{u>0\} \\
      \partial\{u>0\} \setminus \{0\} \text{ is locally smooth} \\
         |\nabla_c u|=1  \text{ on } \partial \{u>0\}\setminus \{0\},
   \end{cases}
   \end{equation}  
  where $\nabla_c$ is the gradient on $\C$ from the inherited metric, and 
  $\Delta_c$ is the Laplace-Beltrami on $\C$.  The above class of solutions may seem restrictive; however, not only will all minimizers of \eqref{e:conefunct} 
  be solutions to \eqref{e:phase1}, but also the common notion of viscosity solution as in  \cite{c87,c88,c89} would also be a solution to 
  \eqref{e:phase1} since the cone $\C$ is three-dimensional.

    We consider two main parametrizations of the cone $\C$. Using spherical coordinates we have 
  \begin{equation}  \label{e:spherical}
   y_1 = r \cos \theta \sin \phi, \qquad
    y_2 = r \sin \theta \sin \phi, \qquad
    y_3 = r \cos \phi, \qquad
    y_4 = cr.
  \end{equation}

  Under these coordinates the area form is $\sqrt{\text{det}(g)}=r^2 \sin \phi \sqrt{1+c^2}$. 
  The local coordinates $g^{ij}$ are given by 
\[ 
 \left( 
        \begin{array}{ccc}
          r^{-2} \sin^{-2} \phi  & 0       & 0 \\
          0                             & r^{-2} & 0 \\
          0                             & 0       & (1+c^2)^{-1} 
        \end{array} 
  \right),
\] 
and in these local coordinates we minimize
 \[
  \int_{\Omega} \sqrt{g} g^{ij} u_i u_j + \sqrt{g} \chi_{\{u>0\}}. 
 \]
 Any minimizer will satisfy 
 \[
  \begin{cases}
    \frac{1}{\sqrt{g}} \partial_j (\sqrt{g} g^{ij} u_i) &=0  \quad\text{ in }  \{u>0\} \\
   \sqrt{g} g^{ij} u_iu_j &= \sqrt{g}   \quad\text{ on } \partial\{u>0\}\setminus \{0\},
  \end{cases}
 \]
and so a minimizer of \eqref{e:conefunct} is a solution to \eqref{e:phase1}. We note that the first condition is written out as 
  \begin{equation}  \label{e:beltrami}
  \frac{1}{1+c^2} \left(u_{rr} + \frac{2}{r}u_r  \right) + \frac{1}{r^2}
   \left(\frac{u_{\theta \theta}}{\sin \phi} + \frac{\cos \phi}{\sin \phi} u_{\phi} + u_{\phi \phi} \right) =0
  \end{equation}
  in the set $\{u>0\}$.

 We may also work in the coordinates 
  \begin{equation}  \label{e:projectmetric}
    y_1 = x_1, \qquad
    y_2 = x_2, \qquad
    y_3 = x_3, \qquad
    y_4 = c \sqrt{x_1^2 + x_2^2 + x_3^2}. 
  \end{equation}
 In these coordinates the area form is $\sqrt{g}=\sqrt{1+c^2}$ and the local coordinates $g^{ij}$ are given by 
  \[ 
  \frac{1}{(1+c^2)r^2}
 \left( 
        \begin{array}{ccc}
          r^{2}+c^2(x_2^2+x_3^2)  & -c^2 x_1x_2                       & -c^2 x_1x_3 \\
          -c^2 x_1 x_2                    & r^2 + c^2(x_1^2 + x_3^2)  & -c^2 x_2 x_3 \\
          -c^2 x_1 x_3                    & -c^2 x_2x_3                       & r^2 + c^2(x_1^2 + x_2^2) 
        \end{array} 
  \right).
\]  
 
 We define $\displaystyle \C_r := \{(y_1,y_2,y_3,y_4) \in \C : \sqrt{y_1^2 +y_2^2 + y_3^2} <r\}$. 
 From the regularity results in \cite{pt16}, we have the following regarding the continuity of the minimizer as well as the growth away from a free boundary point at the vertex. 
  \begin{proposition}  \label{p:growth}
   Let $u$ be a minimizer of \eqref{e:conefunct} on $\Omega \subseteq \C$, then $u$ is H\"older continuous inside $\Omega$. 
   Furthermore, if  the vertex $\{0\} \in \partial \{u>0\}$,
   then there exists two constants $C_1,C_2$ depending on $u$ such that 
    \[
         C_1 r \leq \sup_{\C_r} u \leq C_2 r. 
    \]
  \end{proposition}
  
  We also have 
   \begin{proposition}  \label{p:solution}
    Let $u$ be a minimizer of \eqref{e:conefunct} in $\Omega$. Then $u$ is a solution to \eqref{e:phase1}. 
   \end{proposition}
   
   \begin{proof}
    From Proposition \ref{p:growth} minimizers are continuous. Then by considering variations in the positivity set, it follows that $\Delta_c u=0$ in $\{u>0\}$. 
    Furthermore, the coefficients $g^{ij}$ are smooth away from the vertex of the cone. Therefore, when the cone is three-dimensional one may combine
    the results in \cite{fs07} and \cite{cjk04} to conclude that $\partial \{u>0\}$ is smooth away from the vertex. It then follows from the domain variation 
    techniques in \cite{ac81} that the free boundary relation is satisfied away from the vertex so that any minimizer $u$ of \eqref{e:conefunct}
    is a solution to \eqref{e:phase1}. 
   \end{proof}

 We also have the following Weiss-type monotonicity formula. 
  \begin{proposition}  \label{p:weiss}
    Let $u$ be a minimizer to \eqref{e:conefunct} on $\Omega \subseteq \C$ and $\C_{R} \subset \Omega$.  Assume that the vertex $0 \in \partial \{u>0\}$. 
    Then the  functional 
     \[
      W(r,u):= \frac{1}{r^n} \int_{\C_r} |\nabla_c u|^2 + \chi_{\{u>0\}} - \frac{1}{r^{n+1}}\int_{\partial \C_r} u^2
     \] 
     is monotone increasing in $r$ for $r\leq R$. 
     Furthermore, if $0<r_1<r_2\leq R$ and $W(r_1,u)=W(r_2,u)$ if and only if $u$ is homogeneous of degree $1$ on $\C_{r_2} \setminus \C_{r_1}$. 
  \end{proposition}
  The proof of Proposition \ref{p:weiss} relies on a radial domain variation. Since $\C$ may be parametrized by radial spherical coordinates, the usual proof
  will go through. For two-dimensional cones this was shown in \cite{acl15}. The same proof has also been adapted for a more complicated weight 
  (see \cite{a12,ALP}). As a consequence of Propositions \ref{p:growth} and \ref{p:weiss} we have the following
   \begin{proposition}   \label{p:blowup}
    Let $u$ be a minimizer of \eqref{e:conefunct} on $\Omega \subseteq \C$ and assume the vertex $0 \in \partial\{u>0\}$. Then there exists 
    a sequence $r_k \to 0$ and a $1$-homogeneous solution $u_0$ of \eqref{e:phase1} such that $0 \in \partial\{u_0 >0\}$ and 
    the rescaled functions $u_{k}:=u(r_k x)/r_k$ converge uniformly on compact subsets of $\C$ to $u_0$.  Furthermore, $u_0$ will be a minimizer
    of \eqref{e:conefunct} on all $\Omega \Subset \C$. 
   \end{proposition}
   
   \begin{proof}
    The following proof is standard. From Proposition \ref{p:growth}, the H\"older estimates in \cite{pt16}, and the Arela-Ascoli Theorem, 
    there exists $r_k \to 0$ and $u_0$ such that $u_{r_k} \to u_0$ locally uniformly. That $u_0$ is a minimizer of \eqref{e:conefunct} and therefore also a solution to 
    \eqref{e:phase1} is    standard. From the rescaling property $W(\rho r, u)= W(\rho, u_r)$ of the Weiss functional and the monotonicity of $W(r,u)$ it follows that 
     $W(r, u_0) = W(0+, u)$ for all $r$, and so $u_0$ is homogeneous of degree $1$.  
   \end{proof}

\section{A Second Variation Formula}   \label{s:2ndvariation}
 In this section we adapt the ideas in \cite{cjk04} to obtain a second variation formula. From Proposition \ref{p:blowup} we restrict ourselves to the study 
 of $1$-homogeneous solutions $u$ to \eqref{e:phase1}. We denote the positivity set of a $1$-homogeneous solution $u$ to \eqref{e:phase1}
 by $\Gamma := \{u>0\}$. The free boundary $\partial \Gamma$ is a cone and we denote the mean curvature by $H$. 
 The main result in this section is the following Lemma that gives a second variation formula for $1$-homogeneous solutions to \eqref{e:phase1}
 that are also minimizers of \eqref{e:conefunct}. We define 
  $\mathcal{F}_R := \{F \in C^{\infty}(\overline{\Gamma}): F(x)=0 \text{ if }   |x| \notin (R^{-1},R)\}$ and $\mathcal{F}:= \cup \mathcal{F}_R$.
 \begin{lemma}  \label{l:main}
  Let $\C$ be a cone of type \eqref{e:cone3}. Let $u$ be a $1$-homogeneous minimizer of \eqref{e:conefunct}. Let $g^{ij}$ be the local 
  coordinates from \eqref{e:projectmetric}. Then 
   \begin{equation}  \label{e:variat}
    \int_{\partial \Gamma} H F^2 d \sigma \leq \int_{\Gamma}   g^{ij} F_i F_j  \quad \text{ for every } F \in \mathcal{F}
   \end{equation}
 \end{lemma} 
 
 \begin{proof}
 To prove Lemma \ref{l:main} we intentionally follow the structure of the analogous Lemma in \cite{cjk04}, so that the reader may compare. 
   

 We define  $\Omega := \Gamma \cap B_R$ for some fixed $R$, and 
  use the local coordinates $g^{ij}$ as given in \eqref{e:projectmetric}. 
 We first assume that $F \in C_0^{\infty}(\R^3 \setminus \{0\})$ solves $\partial_i (\sqrt{g} g^{ij} F_j)=0$ in $\Gamma$. Since $\Gamma$ is an NTA domain it is also 
 a Twisted H\"older domain of order $1$ so by Theorem 4.5 in \cite{bb91} there is a boundary Harnack inequality for $\Delta_c$ on $\Gamma$. Then there exists a 
 constant $C$ such that $F \leq Cu$ in a neighborhood of the origin. 
 We define
  \[
   \Omega_{\epsilon} = \{x \in \Omega: u(x) > \e F(x)\}.
  \]
 Let $v_{\epsilon}=u-\epsilon F$ on $\overline{\Omega}_{\epsilon}$ and $v_{\epsilon}=0$ on $B_R\setminus \Omega_{\epsilon}$. We then 
 have that $v_{\epsilon}=u$ on $\partial B_R$. Integrating by parts we have 
  \[
   \int_{B_R \cap \{v_{\epsilon}>0\}} g^{ij} \partial_i v_{\e} \partial_{j} v_{\e} = \int_{(\partial B_R) \cap \Gamma} ug^{ij}(u-\e F)_i \nu_j.
  \]
 Since $\sqrt{g}=\sqrt{1+c^2}$ is a constant we may divide by $\sqrt{g}$ when minimizing the functional. We then have  
  \begin{equation}   \label{e:volume}
   \frac{1}{\sqrt{1+c^2}}\left(J(u,B_R) - J(v_{\e},B_R)\right) = \e \int_{(\partial B_R) \cap \Gamma} u g^{ij}F_i \nu_j + \text{vol}(0<u<\e F).
  \end{equation}
We note that in the above equation the volume element is from the flat metric since we have already divided out by $\sqrt{1+c^2}$. 
 Integrating by parts on $\Omega$ we have
 \[
  \begin{aligned}
  \int_{(\partial B_R) \cap \Gamma} u g^{ij}F_i \nu_j &= \int_{(\partial B_R) \cap \Gamma} (u g^{ij}F_i \nu_j - Fg^{ij}u_{i}\nu_{j} ) \\
    &= \int_{\partial (B_R \cap \Gamma)} (u g^{ij}F_i \nu_j - Fg^{ij}u_{i}\nu_{j} ) - \int_{(\partial \Gamma)\cap B_R} F\\
    &= - \int_{(\partial \Gamma) \cap B_R} F.
  \end{aligned}
 \]
 Then we have that 
  \[
   \frac{1}{\sqrt{1+c^2}}\left(J(u,B_R) - J(v_{\e},B_R)\right) = -\e \int_{(\partial \Gamma)\cap B_R} f d \sigma + \text{vol}(0<u<\e F). 
  \]
  
  In \cite{cjk04} three remarks are given. The first remark is 
  
  \begin{remark}  \label{r:1}
   $u_{\nu \nu}=-H$ on $\partial \Gamma$ except at the origin. 
   \end{remark}
   
   \begin{proof}
   Under a rotation we assume our point $P\in \partial\Gamma$ to be $(x_1,0,0)$. Locally near $P$ the free boundary is given by 
    $u(x_1,x_2,\phi(x_1,x_2))=0$. Differentiating with respect to $i,j$ (with $i,j=x_1,x_2$) we obtain 
    \[
     u_i + \phi_i u_{x_3} =0; \qquad u_{ij} + u_{i x_3} \phi_j + \phi_{ij} u_{x_3} + \phi_i u_{x_3 j} + \phi_i \phi_j u_{x_3 x_3} =0. 
     \]
     We now evaluate at $P$ and use that $\phi_i(P)=0$ and $u_{x_3}(P)=-1$ to conclude that 
     \[
      u_{ij}(P)=\phi_{ij}(P), \quad \text{ for } i,j=x_1,x_2.
     \]
     Recalling that $H$ is the mean curvature of $\partial \Gamma$ at $P$ we have that 
   \begin{equation}  \label{e:curve1}
    u_{x_1 x_1} + u_{x_2 x_2} =H.
   \end{equation}
   Now from the local coordinates given in \eqref{e:projectmetric}  for $i,j=x_1,x_2,x_3$ we have 
   \[
    (\partial_j g^{ij}) u_i + g^{ij}u_{ij}=0.
   \]
   Evaluating at $P$ and using that $u_{x_1}(P)=u_{x_2}(P)$ as well as $x_2=x_3$ at $P$ we obtain
   \[
    \frac{1}{1+c^2}u_{x_1 x_1}(P) + u_{x_2 x_2}(P) + u_{x_3 x_3}(P)=0. 
   \]
   We finally note that because $\partial \Gamma$ is a cone, that $(t,0,0) \in \partial \Gamma$ for all $t \geq 0$, so that $u_{x_1 x_1}(P)=0$. 
   Then combining the above equation with \eqref{e:curve1} we obtain that 
   \[
    H = - u_{x_3 x_3}(P).
   \]
  This concludes the proof of Remark \ref{r:1}. 
 \end{proof}
 
 \begin{remark}  \label{r:2}
  $H \geq 0$. In particular, $\R^3 \setminus \Gamma$ is a finite union of convex cones. 
 \end{remark}
 
 Whereas the proofs of Remarks \ref{r:1} and \ref{r:3} are very similar to those found in \cite{cjk04}, the proof of Remark  \ref{r:2} is different
 and requires Lemma  \ref{l:max} from the Appendix. 
 %

 \begin{proof}[Proof of Remark \ref{r:2}]
   We utilize the homogeneity of $u$. Since $u=rf$ and $\Delta_c u=0$ whenever $u>0$ it follows that $\Delta r^{\alpha}f =0$ as long as 
    \[
     \alpha = (-1 + \sqrt{1+8/(1+c^2)})/2. 
    \] 
    Notice that $0<\alpha<1$ for $c>0$. From Lemma \ref{l:max}, we have $|\nabla_{\theta} f|^2$ achieves the maximum on $S^2 \cap \partial \Gamma$. 
    Since $|\nabla_{\theta} f|=1$ everywhere on  $\partial \Gamma$, we conclude that $|\nabla_{\theta} f| \leq1$ in $\Gamma$. 
     Under a rotation we may assume that $(1,0,0)=P \in \partial \Gamma$ and $\phi$ is the outward 
   unit normal from the spherical coordinates given in \eqref{e:spherical}. 
   Then $\partial_{\phi} f_{\phi}^2 = 2f_{\phi} f_{\phi \phi}\geq 0$, and since $f_{\phi}(P)=-1$ we obtain that $f_{\phi \phi}(P)\leq 0$. 
   Then $u_{33}\leq 0$, and it follows that $H \geq 0$.  
   \end{proof}

 \begin{remark}   \label{r:3}
  Let $z: U \to \R^3$ defined on an open subset $U$ or $\R^2$ be a local parametrization of the surface $\partial \Gamma$ and let $\nu = \nu(s)$ be the
  unit normal to $\partial \Gamma$ at $z(s)$ pointing away from $\Gamma$. In the coordinate system $x(s,t)=z(s)-t\nu(s)$, the volume element 
  $dV= (1+tH + O(t^2))d\sigma(s)dt$, where $d\sigma(s)$ is the $n-1$ area element on the surface.  
 \end{remark}
 The volume element $dV$ in Remark \ref{r:3} is for the flat metric; therefore, Remark \ref{r:3} is identical to that in \cite{cjk04} and the same proof applies. 
 
 We now finish the proof of Lemma \ref{l:main}. Note that in the local coordinates \eqref{e:projectmetric}, if $P \in \partial \Gamma$, then $F F_{\nu}=g^{ij}F_i\nu_j$ at $P$ where $\nu$ is the outward 
 unit normal at $P$. This is most easily seen by under a rotation letting $P=(x_1,0,0)$. Now by combining Remarks \ref{r:1}, \ref{r:2}, and \ref{r:3} as in \cite{cjk04} we obtain
  \[
   \frac{1}{\sqrt{1+c^2}}\left(J(u,B) - J(v_{\e},B)\right)= \epsilon^2 \int_{\partial \Gamma \cap B} (F^2H - FF_{\nu}) d \sigma + O(\epsilon^3). 
  \]
 Since $u$ is a minimizer of $J$ we have that 
 \[
 \int_{\partial \Gamma \cap B} (F^2H - FF_{\nu}) d \sigma \leq 0,
 \]
 so that 
 \[
 \int_{\partial \Gamma \cap B} F^2 H d \sigma \leq \int_{\partial \Gamma \cap B}  FF_{\nu} d \sigma\ 
 = \int_{\partial \Gamma \cap B} g^{ij}F_i\nu_j d \sigma =\int_{\Gamma} g^{ij} F_iF_jd \sigma. 
 \]
 Since the above inequality is true for all $\Delta_c F=0$ it follows that 
 \[
 \int_{\partial \Gamma \cap B} F^2 H d \sigma \leq \int_{\Gamma} g^{ij} F_iF_jd \sigma,
 \]
  for any $F \in C_{0}^{\infty}(\R^3 \setminus \{0\})$. This concludes the proof of Lemma \ref{l:main}. 
  \end{proof}

  We recall the following Lemma from \cite{cjk04}. 
   \begin{lemma}  \label{l:gauss}
    Suppose that $\Gamma$ is an open cone in $\R^3$ and $\partial \gamma = (\partial \Gamma)\cap S^2$ is a finite untion of smooth curves. 
    Then the mean curvature $H$ of $\partial \Gamma$ can be written as 
     \[
        H=\frac{1}{r}\kappa(x/r), \quad r=|x|     
     \]
     where $\kappa$ is the geodesic curvature of the curve $\gamma$ in the unit sphere $S^2$. 
   \end{lemma}
  
  Using Lemma \ref{l:gauss} we may prove 
   \begin{lemma}   \label{l:single}
     Let $\Gamma$ be a cone in $\R^3$ with the mean curvature of $H$ of $\partial \Gamma$ satisfying $H \geq 0$ as well as \eqref{e:variat}
     for every $F \in C_0^{\infty}(\R^3\setminus \{0\})$. Then $\Gamma^c$ has one connected component and is a convex cone contained in a half space. 
   \end{lemma}
  
  \begin{proof}
   We choose $F$ to be a radial function. We let $\gamma:= \partial \Gamma \cap S^2$ and $U := S^2 \cap \Gamma$. Then 
    \[
     \int_{\gamma}\int_0^{\infty} \kappa(\alpha(s))F^2(r) \ dr \ ds \leq \int_{\Gamma} g^{ij} F_i F_j,
    \]
    where $\alpha(s)$ is a parametrization of $\gamma$ with respect to arc length. 
    We now convert the integral over $\Gamma$ from the local coordinates given by \eqref{e:projectmetric} to the spherical coordinates given by \eqref{e:spherical}. 
    \[
     \begin{aligned}
     \int_{\Gamma} g^{ij} F_i F_j &= \frac{1}{\sqrt{1+c^2}}\int_{\Gamma} \sqrt{g} g^{ij} F_i F_j  \\
                                                  &=\frac{1}{\sqrt{1+c^2}}\int_{\C} |\nabla_c F|^2 \\
                                                  &=  \frac{1}{\sqrt{1+c^2}}\int_0^{\infty}\int_{U}r^2 \sin \phi \sqrt{1+c^2}(1+c^2)^{-1}F_r F \\
                                                  &=\frac{|U|}{1+c^2}\int_0^{\infty}r^2 F_r F ,
      \end{aligned}  
     \]
     and $|U|$ is the area of $U$ in the unit sphere. Using the same choice of radial function $F(|r|)$ as in \cite{cjk04}, and combining the above two inequalities
     we obtain 
      \[
       \int_{\gamma} \kappa \ ds \leq \frac{1}{4(1+c^2)} |U|. 
      \]
      We now label $V_j, \ j=1,\ldots,m$ as the connected components of $S^2\setminus U$ and $\gamma_j$ as the boundary curves.
       Using that $H\geq 0$, so that $\kappa \geq0$, we apply the
      Gauss-Bonnet formula 
       \[
        |V_j| + \int_{\gamma_j} \kappa \ ds = 2\pi,
       \] 
       and sum over $j$ to obtain
       \[
       2m\pi - \sum_{j=1}^m |V_j| = \sum_{j=1}^m \int_{\gamma_j} \kappa \ ds = \int_{\gamma} \kappa \ ds \leq \frac{1}{4(1+c^2)} |U|.
       \]
       Since $\displaystyle \sum_{j=1}^m |V_j| = 4 \pi -  |U| $, we have that
       \[
        0< \left(1-\frac{1}{4(1+c^2)} \right) |U| \leq 4\pi - 2m\pi.
       \]
       Then $m=1$, and $U^c$ is a single connected convex component of $S^2$, and so $U^c$ must be contained in a half space. 
  \end{proof}
  
  In this Section we have closely followed the ideas in \cite{cjk04}. Moving forward, however, the ideas in this paper are very different. When $c=0$ one may 
  use a homogeneity argument to conclude that $\partial \Gamma$ is flat. When $c>0$ we will see in the next Section that there exists 
  a stable solution $\Phi_c$ 
  where $\partial \Gamma$ is not flat. Moreover, in Section \ref{s:min} we show that for $c>0$ and small enough that these candidate solutions are indeed minimizers.

  \section{The Symmetric  Solution}    \label{s:symmetric}
    In this Section we present a homogeneous solution $\Phi_c$ which will turn out to be stable and even a minimizer for certain values of $c$. We
    also show that up to rotation $\Phi_c$ is the only possible $1$-homogeneous stable solution. 
    From Lemma \ref{l:gauss} if $u$ is a $1$-homogeneous solution that is stable, then $\{u=0\}$ consists of a single connected component that is convex and 
    contained in a half space. For each $c>0$ we now describe a symmetric candidate solution. Using the spherical coordinates 
    \eqref{e:spherical}, if $\Delta_c r^{\beta} f=0$ and $f(\theta,\phi)$ is a function of $\phi$ alone and independent of $\theta$, then 
     \begin{equation} \label{e:beta}
      \frac{\beta(\beta+1)}{1+c^2} f(\phi) + \frac{\cos \phi}{\sin \phi} f'(\phi) + f''(\phi)=0. 
     \end{equation}
     Under the change of variables $t=\cos \phi$, the function $f$ is a Legendre function and well understood. 
     For $c\geq0$ and fixed $\beta=1$ there is a unique solution 
     with $f'(\phi)=0$ and $f'(\phi_0)=1$ where $f(\phi_0)=0$ which we will denote by $f_{1,c}$. We note that  $f_{1,c}(\phi)\geq0$ for $\phi\leq \pi/2$. 
     Then $rf_{1,c}(\phi)$ is a symmetric (in the variable $\theta$) candidate solution 
     which we will denote by $\Phi_c(r,\phi)=rf_{1,c}(\phi)$. We now show that if $u$ is a stable $1$-homogeneous solution, then up to rotation $u\equiv \Phi_c$.

 \begin{lemma}   \label{l:alexandrov}
   Let $u$ be $1$-homogeneous solution to \eqref{e:phase1} and assume that $\{u=0\}$ is a single connected component contained in a half space. Then up to rotation 
   $u \equiv \Phi_c$. 
 \end{lemma}
 
 \begin{proof}
  We will use the moving plane method as presented in \cite{s71}. By rotation we will assume that $\{u=0\}$ is contained in the hemisphere 
  given by $0\leq \phi \leq \pi/2$. Let $u=rg(\theta,\phi)$. We claim that for $0\leq \phi \leq \pi/2$ we have $g(\theta,\phi)\leq g(\theta,\pi-\phi)$. 
  That is when $g$ is reflected across the equator 
  the bottom half is always greater than or equal to the upper half. Let $h(\theta,\phi)=g(\theta,\pi-\phi)$ and consider $\{g>h\} \cap \{0 \leq \phi \leq \pi/2\}$. 
  Notice that $(g-h)(\theta,\phi/2)=0$ and is a nonnegative eigenfunction $\Delta_{\theta} (g-h)=-\lambda (g-h)$ on $\{g>h\}\cap \{0 \leq \phi\leq \pi/2\}$. Now
  $|\{g-h\leq 0\}\cap \{0\leq \phi\leq \pi/2\}|>0$, so that 
  \[
     \{g-h>0\}\cap \{\phi\leq \pi/2\} \subset \{h>0\}\quad \text{ and } \quad  |\{g-h>0\}\cap \{\phi\leq \pi/2\}|<|\{h>0\}|. 
    \]
     Both $g-h$ and $h$ are both positive 
  eigenfunctions on their respective positivity sets with the same eigenvalue. If $\{g-h>0\}\cap \{\phi\leq \pi/2\}\neq \emptyset$, then the eigenvalue for $g-h$ would be strictly larger
  than the eigenvalue for $h$ which would be a contradiction. Therefore, $g\leq h$ whenever $\phi\leq \pi/2$.

  With this comparison principle in place one can begin to rotate the equator into the set $\{g=0\}$. We have the same comparison argument as before, so that the reflection
  will always lie below. Then conditions $(A),(B),(C),(D)$ of Section 3  in 
  \cite{s71} are all met, so that we conclude $\{g=0\}$ is a spherical cap.  Now the nonnegative eigenfunction on a domain is unique, and after rotation there is a nonnegative
  eigenfunction given by \eqref{e:beta}. Since $u=rg$ and $\Delta_c u=0$ it follows that $\beta=1$ so that $g=f_{1,c}$, so that $u=\Phi_{c}$.
  
  \end{proof}

\section{Stable Solutions}  \label{s:stable}  
 In this Section we prove Theorem \ref{t:main1}. We first show that for $c$ large enough the solution $\Phi_c$ is not stable. 
 
 \begin{lemma}  \label{l:big}
  There exists $M <\infty$ such that if $c \geq M$, then $\Phi_{c}$ is not a stable solution.  
  \end{lemma}
 
 \begin{proof}
 Let $\phi_0$ be such that $f_{1,c}(\phi_0)=0$. 
  We also label $t_0 = \cos \phi_0$. We have $t_0 \to -1$ as $c \to \infty$. Furthermore, the mean curvature $H$ of $\partial \Gamma$ is $H_1/r$ where 
 $r$ is the distance from the origin, and $H_1$ is the mean curvature at radius $1$ which is  
  \[
   H_1 = \frac{|t_0|}{\sqrt{1-t_0^2}}. 
  \]
 We choose a radial function $F(r)$ which is smooth and compactly supported in $B_1 \setminus \{0\}$. 
 Then 
  \[
   \begin{aligned}
  \int_{\partial \Gamma} H F^2 \ d \sigma &= \int_{0}^{|t_0|} H 2 \pi \frac{1+t_0^2}{|t_0|}r F^2(t/t_0)\sqrt{1 + (1-t_0^2)/t_0^2} \ dr  \\
                                                                &= \int_{0}^{|t_0|}  2 \pi  F^2(t/t_0)|t_0|^{-1}\ dr \\ 
                                                                &= \int_0^1 F^2(r) \ dr. 
   \end{aligned}
  \]
  Thus, the above quantity remains constant independent of $c$. 
  Furthermore, by converting the local coordinates from \eqref{e:projectmetric} to the spherical coordinates \eqref{e:spherical} we have 
  \[
   \begin{aligned}
   \int_{\Gamma} g^{ij} F_i F_j &= \frac{1}{\sqrt{1+c^2}}\int_{\C} |\nabla_c F|^2 \\
                                             &= \frac{|S^2 \cap \{f_{1,c}>0\}|}{1+c^2} \int_{0}^{1} r^2 [F'(r)]^2 \ dr \\
                                             &=  \frac{2 \pi [2/3 - |t_0| + |t_0|^3/3]}{1+c^2} \int_0^1 r^2 [f'(r)]^2 \ dr. 
   \end{aligned}
  \]
As $c \to \infty$, the above quantity goes to zero. Therefore, for large enough $c$ the second variational inequality \eqref{e:variat} fails, 
 and we conclude that our candidate solution $\Phi_c$ is not stable. 
\end{proof}

  For this next Lemma we define $G_{\beta}(r,\phi):=r^{\beta}g_c(\phi)$ where $g_c(\phi)$
 is a solution to \eqref{e:beta} with $-1<\beta<0$ and the conditions 
   $g_c'(0)=0$ and $g_c(\phi) \geq 0$ for all $0\leq \phi < \pi$. 
   Now $G_{\beta}' \geq 0$ as well. We will see in the
   proof that it is convenient to choose $\beta=-1/2$ in which case we will simply write $G$. Also, when the value of $c$ is fixed we will write simply $g(\phi)$. The
   function $g(\phi)$ should not be confused with the local coordinates $g^{ij}$ from the metric defined by \eqref{e:projectmetric}. 
 \begin{lemma}   \label{l:stable}
  Let $\Phi_c$ be the symmetric solution as described in Section \ref{s:symmetric}. Let $\Gamma=\{\Phi>0\}$, and $\phi_0$ 
  such that $\Phi_c(\phi_0)=0$. Then \eqref{e:variat} is equivalent to 
   \begin{equation}   \label{e:equiv}
    H_1 \leq \frac{g'(\phi_0)}{g(\phi_0)}
   \end{equation}
   where $H_1$ is the mean curvature of $\Gamma$ at $r=1$, and $g(\phi)$ is a nonnegative solution to \eqref{e:beta} with
   \[
    \beta=-1/2, \quad g'(0)=0, \quad g(\phi) >0, \quad g'(\phi)\geq 0. 
   \]
 \end{lemma}
 
 \begin{proof}
  We first show that \eqref{e:equiv} implies \eqref{e:variat}. Let $F \in \mathcal{F}$ with  $F=0$ on $\partial B_{R_2} \cap \Gamma$ and $\partial B_{R_1} \cap \Gamma$. We aim to minimize the quantity 
   \[
     E(F,\Gamma):= \frac{\int_{\Gamma} g^{ij} F_i F_j}{  \int_{\partial \Gamma} H F^2 d \sigma} 
   \]
  From the standard theory of Calculus of variations, a unique minimizer $v$ will exist and satisfy the Steklov eigenvalue problem  
   \[
    \begin{cases}
      \partial_j (g^{ij}v_i) =0 &\text{ in } \Gamma \cap (B_{R_2}\setminus B_{R_1}) \\    
      g^{ij} v_i \nu_j = \lambda H v &\text{ on }  \partial \Gamma \cap  (B_{R_2}\setminus B_{R_1}),
    \end{cases}
   \]
   where $\lambda = E(v,\Gamma)$. In order to obtain a lower bound for $\lambda$, we consider the functions 
   $G_{\beta}$.
   Since $G_{\beta} >0$ on 
   $\partial \Gamma \cap  (B_{R_2}\setminus B_{R_1})$, there exists some constant $M$ and a point 
   $x_0 \in \overline{\Gamma \cap (B_{R_2}\setminus B_{R_1})}$ such that $M G_{\beta} \geq v$ and $M G_{\beta}(x_0)= v(x_0)$. 
   Since $\Delta_c G_{\beta}=0$ in $\Gamma$, then from the maximum principle we conclude $x_0 \in \partial \Gamma \cap (B_{R_2}\setminus B_{R_1})$. 
   Then if $x_0=(r_0,\theta_0,\phi_0)$ and $H_1$ is the mean curvature of $\Gamma$ at $x_0/|r|$ we obtain 
    \[
       \begin{aligned}
       r^{\beta}\frac{Mg'(\phi_0)}{r} &= g^{ij} \partial_i M G_{\beta}  \nu_j(x_0) \\
                                   &\leq g^{ij}(x_0) v_i \nu_j \\
                                   &= \lambda H v(x_0) \\
                                   &= r^{\beta}\lambda H_1 \frac{Mg(x_0)}{r}.
        \end{aligned}
    \]
    Thus
    \[
     \frac{g'(\phi_0)}{g(\phi_0)} \leq \lambda H_1 = E(v,\Gamma) H_1. 
    \]
    Then if 
     \[
      1 \leq \frac{1}{H_1} \frac{g'(\phi_0)}{g(\phi_0)}, 
     \]
    then also $E(v,\Gamma)\geq 1$, and so \eqref{e:variat} holds.

    We now show that \eqref{e:variat} implies \eqref{e:equiv}. 
    From the symmetry of $\beta$ for $-1<\beta<0$ in \eqref{e:beta}, one may expect a maximum bound from below when $\beta=-1/2$. We now show this is indeed the case.
    Recall that  $G(r,\phi)=r^{-1/2}g(\phi)$ where $g(\phi)$ is a solution to \eqref{e:beta} with $\beta=-1/2$.     
    Then in local coordinates \eqref{e:projectmetric} we have 
    \[
     \begin{aligned}
     \int_{\Gamma \cap (B_{R_2}\setminus B_{R_1})} g^{ij} G_iG_j 
       &= \int_{\partial \Gamma} g^{ij}G_i \nu_{j} 
       + \int_{\partial B_{R_2}\cap \Gamma} g^{ij} G_i \nu_j +  \int_{\partial B_{R_1}\cap \Gamma} g^{ij} G_i \nu_j \\
       &= \int_{\partial \Gamma} g^{ij}G_i \nu_{j},       
     \end{aligned}
    \]
    since 
    \[
     \int_{\partial B_{R_2}\cap \Gamma} g^{ij} GG_i \nu_j +  \int_{\partial B_{R_1}\cap \Gamma} g^{ij} G_i \nu_j=0,
    \]
    because $G(r,\phi)=r^{-1/2}g(\phi)$. This is also why we chose $\beta=-1/2$. Now
    \[
     \int_{\partial \Gamma} g^{ij}GG_i \nu_{j} = \int_{\partial \Gamma} r^{-1/2}g(\phi_0) r^{-1/2} g'(\phi_0)
     = \int_{\partial \Gamma} r^{-1}g(\phi_0)g'(\phi_0). 
    \]
    We now define 
     \[
      \Psi_1 := g(\phi)[2R_1^{-3/2}(r-R_{1}/2)]_+,
     \]
     and note that 
     \[
      \int_{\Gamma\cap B_{R_1}} g^{ij}\partial_i \Psi_1 \partial_j \Psi_1 \leq C_1 \int_0^{R_1} r^2 R_1^{-3} \leq C_1. 
     \]
     Similarly if 
     \[
      \Psi_2 := -g(\phi) R_2^{-3/2}[r-2R]_+,
     \]
     then 
     \[
      \int_{\Gamma \setminus B_{R_2}} g^{ij}\partial_i \Psi_2 \partial_j \Psi_2 \leq C_2 \int_{0}^{2R_2} r^2 R_2^{-3} \leq C_2.
     \]
     Now if we define 
     \[
      \tilde{F}(x)
      :=
      \begin{cases}
       \Psi_1(x) &\text{ if } r< R_1 \\
       G(x)  &\text{ if } R_1 \leq r \leq R_2 \\
       \Psi_2(x)  &\text{ if } r > R_2,
      \end{cases}
     \]
     then 
     \[
     \begin{aligned}
      \frac{1}{H_1}\frac{g'(\phi_0)}{g(\phi_0)} 
      &\leq E(\tilde{F},\Gamma) \\
      &\leq \frac{\int_{\Gamma}g^{ij} \tilde{F}_i \tilde{F}_j}{\int_{\partial \Gamma \cap (B_{R_2}\setminus B_{R_1})}Hr^{-1/2}g^2(\phi_0)} \\
      &\leq \frac{C_1 + C_2}{\int_{\partial \Gamma \cap (B_{R_2}\setminus B_{R_1})}Hr^{-1/2}g^2(\phi_0)}
       + \frac{\int_{\Gamma}g^{ij} \tilde{F}_i \tilde{F}_j}{\int_{\partial \Gamma \cap (B_{R_2}\setminus B_{R_1})}Hr^{-1/2}g^2(\phi_0)} \\
       &= \frac{C_1 + C_2}{\int_{\partial \Gamma \cap (B_{R_2}\setminus B_{R_1})}Hr^{-1/2}g^2(\phi_0)}
             + \frac{g'(\phi_0)g(\phi_0)}{H_1g^2{\phi_0}}.
      \end{aligned}
     \]
     Now as $R_1 \to 0$ and $R_2 \to \infty$ the first term in the last line above goes to zero. Hence, we conclude that 
      \[
        \min E(F,\Gamma) = \frac{1}{H_1}\frac{g'(\phi_0)}{g(\phi_0)}. 
      \]
      Then \eqref{e:variat} holds if and only if $H_1 \leq g'(\phi_0)/g(\phi_0)$. 
 \end{proof}

\begin{corollary}   \label{c:stable}
 There exists $\e_0>0$ such that if $c\leq \e_0$, then $\Phi_c$ is a stable solution; i.e., if $\Gamma =\{\Phi_c>0\}$, then \eqref{e:variat} holds. 
\end{corollary} 
 
\begin{proof}
 From Lemma \eqref{e:equiv} we have that \eqref{e:variat} is holds if and only if 
  \[
   H_1 \leq \frac{g_c'(\phi_0)}{g_c(\phi_0)}.
  \]
  As $c\to 0$ we have that $H_1 \to 0$,  $\pi/2 \leq \phi_0 \to \pi/2$, and $g_c \to \cos \phi$. Then the above inequality will be satisfied for small enough $c$. 
\end{proof}

We are now ready to prove our first Main Theorem. 
\begin{proof}[Proof of Theorem \ref{t:main1}]
 We first normalize by letting $g_c(0)=1$. If $0\leq c_1<c_2$, the $g_{c_1}$ is a supersolution to \eqref{e:beta} for $c_2$, and 
 so $g_{c_1}>g_{c_2}$ on $(0,\pi)$. Furthermore, for fixed $\phi_0$ if $g_{c_1}(\phi_0)-h=g_{c_2}(\phi_0)$, then $g_{c_2}-h$ is a subsolution 
 on the interval $(\phi_0,\pi)$ to \eqref{e:beta} for 
 $c_2$. Then $g_{c_1}-h>g_{c_2}$ on $(\phi_0,\pi)$. It follows that $g_{c_1}'>g_{c_2}'$ on $(0,\pi)$. Then $g_{c_1}/g_{c_2}$ is increasing on 
 $(0,\pi)$, so that $\log (g_{c_2}/g_{c_1})$ is increasing on $(0,\pi)$. Then 
  \begin{equation}  \label{e:ephi}
   \frac{g_{c_1}'}{g_{c_1}} \geq \frac{g_{c_2}'}{g_{c_2}} \text{  on } (0,\pi) \text{  if } 0\leq c_1 < c_2. 
  \end{equation}
 Now if $g_{c_2}(\phi_2)=0$, then the mean curvature of $\Gamma$ is 
  \[
   H_1 = -\frac{\cos \phi_2}{\sin \phi_2}. 
  \]
 Then \eqref{e:variat} which is equivalent to \eqref{e:equiv} holds if and only if 
  \begin{equation}   \label{e:phinew}
   1 \leq -\frac{\sin\phi_2}{\cos\phi_2} \frac{g_{c_2}'(\phi_2)}{g_{c_2}(\phi_2)}. 
  \end{equation}
 
 From Corollary \ref{c:stable} there exists $c_2 >0$ such that $\Phi_{c_2}$ is a stable solution, so that  \eqref{e:phinew} holds where
 $g_{c_2}(\phi_2)=0$. We seek
 to show that if $0\leq c_1<c_2$, then \eqref{e:phinew} hols for $c_1$ where $g_{c_1}(\phi_1)=0$. 
 We take the derivative 
  \[
   \begin{aligned}
   \frac{d}{d\phi} \left(-\frac{\sin\phi}{\cos\phi} \frac{g_{c_2}'(\phi)}{g_{c_2}(\phi)} \right)
    &= -\frac{1}{\cos^2\phi} \frac{g_{c_2}'(\phi)}{g_{c_2}(\phi)} \\
     & \quad -\frac{\sin \phi}{\cos \phi} \left[\frac{g_{c_2}(\phi)g_{c_2}''(\phi)-[g_{c_2}'(\phi)]}{g_{c_2}^2(\phi)} \right] \\
     &=  -\frac{1}{\cos^2\phi} \frac{g_{c_2}'(\phi)}{g_{c_2}(\phi)} \\ 
     & \quad +\frac{g_{c_2}'(\phi)}{g_{c_2}(\phi)} - \frac{\sin \phi}{\cos \phi} \frac{1}{4(1+c^2)} 
        + \frac{\sin \phi}{\cos \phi} \frac{[g_{c_2}'(\phi)]^2}{g_{c_2}^2(\phi)} \\
      &=\frac{\sin \phi}{\cos \phi} \left[ -\frac{\sin \phi}{\cos \phi} \frac{g_{c_2}'(\phi)}{g_{c_2}(\phi)}
             +\frac{[g_{c_2}'(\phi)]}{g_{c_2}^{2}(\phi)} - \frac{1}{4(1+c^2)}\right].   
   \end{aligned}
  \]
 Since \eqref{e:phinew} holds at $\phi_2$, and since $\phi_2 >\pi/2$, it follows that in a small neighborhood around $(\phi_2 - \e, \phi_2 +\e)$
 that the above derivative is negative. Then if $c_3 \in (c_2 - \delta, c_2+\delta) $ for small enough delta, then $\phi_{3} \in (\phi_2 - \e, \phi_2 + \e)$
 where $g_{c_3}(\phi_3)=0$. Then if $c_3 \in (c_2 - \delta, c_2)$ we have that
  \[
   1 \leq-\frac{\sin\phi_2}{\cos\phi_2} \frac{g_{c_2}'(\phi_2)}{g_{c_2}(\phi_2)} 
      \leq -\frac{\sin\phi_3}{\cos\phi_3} \frac{g_{c_2}'(\phi_3)}{g_{c_2}(\phi_3)}
      \leq -\frac{\sin\phi_3}{\cos\phi_3} \frac{g_{c_3}'(\phi_3)}{g_{c_3}(\phi_3)}.
  \]
 The last inequality follows from \eqref{e:ephi}. We have shown that the set of points $c \in [0,\infty)$ for which \eqref{e:phinew} holds is open to the left. 
 Since the inequality is preserved in a limit, it follows that the set of points $c \in [0,\infty)$ for which \eqref{e:phinew} holds is also closed to the left. 
 Combining this with Lemma \ref{l:big}
 there is then a last point $c_0<\infty$ such that $\Phi_c$ is stable if and only if $0\leq c \leq c_0$.  From Corollary \ref{c:stable} we have that $c_0>0$. This concludes the proof. 
\end{proof}

We now give the 
\begin{proof}[Proof of Corollary \ref{c:avoidvertex}]
 Let $u$ be a minimizer of \eqref{e:conefunct} with $c>c_0$ with $c_0$ given in Theorem \ref{t:main1}. Suppose by way of contradiction that the vertex 
 $0 \in \partial\{u>0\}$. By Proposition \ref{p:blowup} there exists a $1$-homogeneous minimizer $u_0$ of \eqref{e:conefunct}. Then $u_0$ is also stable, and so by 
 Lemmas \ref{l:single} and   \ref{l:alexandrov} we conclude $u_0 \equiv \Phi_c$. But $\Phi_c$ is not stable for $c>c_0$, and we obtain a contradiction.  
\end{proof}

\section{A minimizer for $c>0$.}  \label{s:min}

 In this section we prove that the symmetric solution $\Phi_c$ defined in Section \ref{s:symmetric} is indeed a minimizer for $0\leq c \leq c_0$ for $c_0$ small enough. This is accomplished by trapping 
 $\Phi_c$ between a continuous family of sub- and supersolutions to the free boundary problem \eqref{e:phase1}. 
 This shows that $\Phi_c$ is a unique solution subject to its own boundary data. Since a minimizer does
 exist and is a solution, then $\Phi_c$ is a minimizer. We first construct a continuous family of subsolutions from below.

 \begin{lemma}  \label{l:below1}
  There exists $c_0 >0$ such that if $0 \leq c \leq c_0$, and if $u$ is a solution to \eqref{e:phase1} with $u = \Phi_c$ on $\partial B_1$, then 
  $u \geq \Phi_c$ in $B_1$. 
 \end{lemma} 

 \begin{proof}
  We let  $\Phi_c=rf_{1,c}(\phi)$. For convenience throughout this proof we will simply write $f(\phi)$ in place of 
  $f_{1,c}(\phi)$. We consider $(v_{\e})_+$ where $v_{\epsilon}:=rf(\phi)-\e r^{\beta}g(\phi)$ and
   \[
    g(\phi):=M- \cos(\phi).
   \] 
  Notice that
   \[
   \Delta_c r^{\beta}g(\phi) = r^{\beta-2}\left[\frac{\beta(\beta+1)}{1+c^2}(M-\cos(\phi))+ 2\cos(\phi)\right]
   \]
   By choosing $-1<\beta<0$ and $M$ large enough depending on $\beta$ and $c_0$, then $\Delta_c r^{\beta}g(\phi) \leq 0$ for $0 \leq c\leq c_0$. 
   Thus,  $\Delta_c (v_{\e})_+ \geq 0$ independent of $\e$. For convenience throughout the remainder of the proof we will write simply $v$ in place of $v_{\e}$.

   We note that 
    \[
     \begin{aligned}
      v_r &= f(\phi)-\e\beta r^{\beta-1}g(\phi). \\
      v_{\phi}&=rf'(\phi)-\e r^{\beta} g'(\phi). 
     \end{aligned}
    \] 
   Furthermore, on  $\{v=0\}$ we have $rf(\phi)=\e r^{\beta} g(\phi)$, so that on $\{v=0\}$ we obtain 
    \begin{equation}   \label{e:express}
     | \nabla_c v |^2 = \frac{1}{1+c^2} v_r^2 +\frac{1}{r^2} v_{\phi}^2 
     = \frac{(1-\beta)^2}{1+c^2} f^2(\phi) + \left[f'(\phi)-f(\phi) \frac{g'(\phi)}{g(\phi)} \right]^2. 
    \end{equation}
    In order to use a comparison principle, we need $|\nabla_c v |^2 > 1$ on $\partial \{v>0\}$. We let $\phi_0$ be such that $f(\phi_0)=0$. Notice that 
    $\phi_0 > \pi/2$. Furthermore, since $\e r^{\beta}g(\phi)\geq 0$, then $v(r,\phi)=0$ only when $\phi < \phi_0$. Finally, 
    we note that 
     \begin{equation}  \label{e:subderiv}
      |\nabla_{c} v(r, \phi_0)|^2 = \left[f'(\phi_0)-f(\phi_0) \frac{g'(\phi_0)}{g(\phi_0)} \right]^2 =  |f'(\phi_0)|^2 = |\nabla_c \Phi_c(r,\phi_0) |^2 =1.     
     \end{equation}
     We now take the derivative in $\phi$ of the last expression in \eqref{e:express}. 
     \[
      \begin{aligned}
      & \frac{d}{d \phi}\left( \frac{(1-\beta)^2}{1+c^2} f^2(\phi) + \left[f'(\phi)-f(\phi) \frac{g'(\phi)}{g(\phi)} \right]^2 \right)   \\
      &= \frac{2(1-\beta)^2}{1+c^2} f(\phi)f'(\phi) + 2\left[f'(\phi) -f(\phi)\frac{g'(\phi)}{g(\phi)}\right] \times \\
        & \quad \left[f''(\phi) - f'(\phi)\frac{g'(\phi)}{g(\phi)} - f(\phi) \left(\frac{g(\phi)g''(\phi)-[g'(\phi)]^2}{g^2(\phi)} \right) \right].   
      \end{aligned}
     \] 
     We recall that 
      \[
       f''(\phi)= -\frac{\cos \phi}{\sin \phi} f'(\phi) - \frac{2}{1+c^2} f(\phi). 
      \]
     Substituting this into the computed derivative, reorganizing terms, and dividing by $2$, we obtain that the derivative (divided by $2$)
     is the sum of the following three pieces 
      \begin{equation}  \label{e:three}
       \begin{aligned}
        &\left[\frac{(1-\beta)^2}{1+c^2}-\frac{2}{1+c^2}-\left(\frac{g(\phi)g''(\phi)-[g'(\phi)]^2}{g^2(\phi)}\right) \right]f(\phi)f'(\phi) \\
        &\quad +f^2(\phi)\frac{g'(\phi)}{g(\phi)} \left[\frac{2}{1+c^2} + \left(\frac{g(\phi)g''(\phi)-[g'(\phi)]^2}{g^2(\phi)} \right) \right] \\
        &\quad +\left[f'(\phi)-f(\phi) \frac{g'(\phi)}{g(\phi)} \right] \left[\frac{\cos \phi}{\sin \phi} + \frac{g'(\phi)}{g(\phi)} \right](-f'(\phi)) \\
        &= I + II + III.
       \end{aligned}
      \end{equation}
     We choose $\beta=-1/2$. We also have that $f''(\phi)<-\delta_1 <0$ for $0\leq \phi \leq \phi_0 $ and some $\delta_1>0$ depending on $c_0$ 
     and independent of $c$ if $0\leq c \leq c_0$. 
     Then choosing $M$ large enough depending on $c_0$, there exists a constant $\delta_2$ depending on $c_0$
     such that for $0\leq c\leq c_0$ we have 
      \[
       f(\phi)\frac{g'(\phi)}{g(\phi)} \leq \delta_2 f'(\phi) \text{ for } \phi \geq \phi_0.
      \]
      Then 
      \[
       I + II \leq \left[\frac{(1-\beta)^2}{1+c^2}-\frac{2+\delta_2}{1+c^2}-(1+\delta_2)\left(\frac{g(\phi)g''(\phi)-[g'(\phi)]^2}{g^2(\phi)}\right) \right]f(\phi)f'(\phi).
      \]
      Choosing again $M$ large enough, there exists $C_1$ such that if $0 \leq c \leq c_0$, then 
      \[
      I + II \leq C_1 f(\phi)f'(\phi) \leq  0 \quad \text{ for } 0\leq  \phi \leq \phi_0. 
      \]
       The parameter $M$ is now fixed. For $c_0$ small, the angle $\phi_0$ is close to $\pi/2$. Therefore,  
       to control $III$ we choose $c_0$ small enough so that 
        \[
         \frac{\cos \phi}{\sin \phi} + \frac{g'(\phi)}{g(\phi)} > 0 \quad \text{ for } \phi\leq \phi_0. 
        \] 
       Then $III \leq 0$, and so  $I + II + III \leq 0$. Furthermore, in the above proof it is clear that $I + II + III<0$ when 
       $0<\phi<\phi_0$, so that  $|\nabla_c v|^2 >1$ on $\partial \{v=0\}$. 
       Thus we have shown that $\Delta (v_\e)_{+} \geq 0$ and $|\nabla (v_\e)_+|>1$ on $\partial \{(v_{\e})_+>0\}$, and this is independent of $\e>0$.  
              
       Now let $u$ be a solution to \eqref{e:phase1} with $u = \Phi_c$ on $\partial B_1$. Suppose by way of contradiction that 
       there exists $x_0 \in B_1$ such that $u(x_0)< \Phi_c (x_0)$. We have that $v_{\e}<u$ on $\partial B_1$ for all $\e>0$, and
       we may choose $\e $ large enough so that $v_{\e}< u$ in $B_1$. Then $(v_{\e})_+< u$ in $\{u>0\}$. 
       Also, $\{(v_{\e})_+ > 0\} \subset \{u>0\}$. We now continuously shrink $\e$ until either $(v_{\e})_+$ touches $u$ from below in $\{u>0\}$, or 
       $\partial \{(v_{\e})_+>0\}$ touches $\partial \{u>0\}$. Since $(v_{\e})_+ \to \Phi_c$ pointwise on $B_1\setminus \{0\}$, there exists an $\e_0$ and $x_1\neq 0$ such that 
       $(v_{\e})_+ \leq u$ in $B_1$ and either 
       $(v_{\e_0})_+(x_1)=u(x_1)>0$ or $x_1 \in (\partial \{(v_{\e_0})_+>0\}\cap \partial\{u>0\})$. Since $\Delta_c (v_{\e_0})>0$ in 
       $\{(v_{\e_0})_+>0\}$, the first possibility is a violation of the comparison principle. Since $(v_{\e_0})_+ \leq u$, if $(v_{\e_0})_+(x_1)=u(x_1)=0$, then 
       since $|\nabla_c u(x_1)|=1 < |\nabla_c (v_{\e_0})_+|$ we also obtain  a contradiction. 
       Therefore, if $0 \leq c \leq c_0$ and if $u$ is a solution to \eqref{e:phase1} with $u = \Phi_c$ 
       on $\partial B_1$, then $u \geq \Phi_c$. 
     \end{proof}

\begin{lemma}  \label{l:above1}
 There exists $c_0 >0$ such that if $0 \leq c \leq c_0$, and if $u$ is a solution to \eqref{e:phase1} with $u=\Phi_c$ on $\partial B_1$, then $u \leq \Phi_c$. 
 \end{lemma}

\begin{proof}
 The beginning of the proof is similar to the proof of Lemma \ref{l:below1}.  We consider 
 $v_{c}=rf_{1,c}(\phi)+\e r^{\beta} g(\phi)$ with $\beta=-1/2$ and $g(\phi)=M-\cos\phi$. We will write 
 $f$ in 
 place of $f_{1,c}$ when 
 $c$ is understood. We use the subscript $c$ on the function $v_c$ because later in the proof we will let $c$ vary. 
 
 On the set $\{v_c=0\}$ we have 
  $rf(\phi)=-r^{\beta}g(\phi)$, so that once again we obtain that on $\{v_c=0\}$ we have
  \[
    |\nabla_c v_c|^2 = \frac{(1-\beta)^2}{1+c^2} f^2(\phi) + \left[f'(\phi)-f(\phi) \frac{g'(\phi)}{g(\phi)} \right]^2.   
  \]
 One main difference from the proof of Lemma \ref{l:below1} is that now if $v_c(r,\phi)=0$, then $\phi \geq \phi_0$ where 
 $f(\phi_0)=0$. Using the same computations as in the proof of Lemma \ref{l:below1}, we obtain by taking the derivative in $\phi$ that
 $|\nabla_c v_c(r,\phi)|^2 < 1$ provided that $\phi < \phi_1$ where $\phi_1>\phi_0$ and $\phi_1$ is determined by letting $c$ be small so that 
 the third term $III$ in \eqref{e:three} is negative. We now fix $\phi_2$ with $\phi_0<\phi_2 < \phi_1$. Notice that $\phi_0$ depends on $c$, but for small enough $c_0$,
 $\phi_1$ and $\phi_2$ will not depend on $c$ for $0\leq c\leq c_0$. For fixed $\e_0$, let $r_0$ be such that 
  \[
   r_0f(\phi_2)+ \e_0 r_0^{\beta}g(\phi_2) =0. 
  \]
 Then 
  \begin{equation}  \label{e:epsio}
   \e_0 = -r_0^{1-\beta}\frac{f(\phi_2)}{g(\phi_2)}. 
  \end{equation}
  We have that $|\nabla_c v_c|>1$ on $\partial \{v>0\}$ as long as $\phi\leq \phi_2$. We now redefine the function $v_c$ on $B_{r_0}$.   
  We first notice that from \eqref{e:epsio}, we may rescale by 
  \[
   v_{r_0} := \frac{v_c(r_0 x)}{r_0}. 
  \]
  The rescaled function $v_{r_0}$ is defined on $B_{1/r_0}$. Therefore, we may assume without loss of generality, that $\e=-f(\phi_2)/g(\phi_2)$
  and that we are redefining the values on $B_1$. We now define $U:= B_1 \cap \{x_3 > \cos(\phi_2)\}$ and also define
   \[
    \tilde{v}_c:=
     \begin{cases}
      \Delta_c \tilde{v}_c =0 &\text{ in } U \\
      \tilde{v}_c = v &\text{ on } \partial U \cap \{x_3 > \cos(\phi_2)\} \\
      \tilde{v}_c = 0 &\text{ on } \partial U \cap \{x_3 = \cos(\phi_2)\}.  
     \end{cases}
   \]
  Finally, we paste the two functions $\tilde{v}_c$ and $v_c$ by defining 
   \begin{equation}
    w_c:= 
     \begin{cases}
      (v_c)_+ &\text{ in } B_1^c \\
      \tilde{v}_c &\text{ in } U  \\
      0 &\text{ in } B_1 \setminus U.
     \end{cases}
   \end{equation}
  Using a compactness argument, we will show for $\phi_2$ fixed and small enough $c$, that $\Delta_c w_c\leq 0$ in $\{w_c>0\}$ and $|\nabla_c w_c|<1$ on 
  $\partial \{w_c>0\}\cap B_1^c$ and $\partial \{w_c>0\}\cap B_1$. Now $\{w_c>0\}$ is a wedge-type domain at $\partial \{w_c>0\}\cap \partial B_1$, but because
  of the angle of the wedge we will see that $\partial \{w_c>0\}\cap \partial B_1$ can never touch the free boundary $\partial\{u>0\}$ of a solution $u$ to \eqref{e:phase1}.

  We first show that if $c=0$, we obtain the needed properties for $w_0$. Then using a compactness argument, we show that if $c$ is small enough, that $w_c$ will
  also have the needed properties. 
  If $c=0$, then on $\partial B_1$ and $\phi<\phi_2$ we have that 
   \[
    v_0(1,\phi)= \cos \phi - \frac{\cos \phi_2}{g(\phi_2)} g(\phi),  
   \]
   so that 
   \[
    \tilde{v}_0 = r \left[ \cos \phi - \frac{\cos(\phi_2)}{M-\cos \phi_2} (M- \cos \phi)\right].  
    \]
   This is just a linear function, and we notice that 
   \[
    \frac{\partial \tilde{v}_0}{\partial x_3} (x_1,x_2,-\cos \phi_2) = 1 + \frac{\cos(\phi_2)}{M-\cos \phi_2} < 1. 
   \]   
   Furthermore, on $\partial B_1$ with $x_3 > \cos \phi _2$, if $\nu$ is the outward unit normal to $ B_1$ we have  
   \[
    \begin{aligned}
     \frac{d \tilde{v}_c}{d \nu}(1,\phi) &= \cos \phi - \frac{\cos \phi_2}{M- \cos \phi_2} (M-\cos \phi) \\
                                              &> \cos \phi - \beta \frac{\cos \phi_2}{M- \cos \phi_2} \cos \phi \\    
                                              &= \frac{d (v_c)_+}{d \nu}(1,\phi), 
    \end{aligned}
   \]
   as long as $M>2$. Then $\Delta w_0 >0$ weakly in $\{w_0>0\}$. 
    
    Now $U$ is not a $C^1$ domain. However by Lemma \ref{l:Uconv}, if we let $c \to 0$, then $\tilde{v}_c \to \tilde{v}_0$ in $C^1$ on $\overline{U}$. 
   Thus, for $c$ small enough we obtain that 
   \[
    |\nabla_c \tilde{v}_c| < 1 \text { on } \{x_3 = \cos \phi_2\}\cap B_1
   \]
  and 
   \[
    \frac{d \tilde{v}_c}{d \nu}(1,\phi) - \frac{d (v_c)_+}{d \nu}(1,\phi) > 0.
   \]
  Thus, for $c_0$ small enough and for $0\leq c\leq c_0$,  we have $\Delta_c w_c < 0$ weakly in $\{w_c >0\}$,  
  and a comparison principle holds.
  
  We now let $c_0$ be chosen as above. Let  $0 \leq c \leq c_0$ and let $u$ be a solution to \eqref{e:phase1} with $u = \Phi_c$ 
  on $\partial B_{\rho}$ with $\rho < - \cos \phi_2$ where $\phi_2$ defines $U$. Suppose that there exists $x_0 \in B_{\rho}$ such 
  that $u(x)>\Phi_c(x)$. We may choose $\e$ large enough so that $w_{\e}> u$ on $\overline{B}_{\rho}$. Since $w_{\e}>\Phi_c$ in $B_1$ for every $\e>0$, 
  then also $w_{\e}> u$ on $\partial B_{\rho}$ for every $\e>0$. By continuously moving $\e$ towards $0$, there exists $\e_1$ and $x_1\neq 0$ such that $w_{\e_1}\geq u$ 
  and either 
  $w_{\e_1}(x_1) = u(x_1)>0$ or $x_1 \in \partial\{w_{\e_1}>0\} \cap \{u>0\}$. If  $w_{\e_1}(x_1) = u(x_1)>0$, we obtain a contradiction since 
  $\Delta_c w_{\e_1}<0$ weakly in $\{w_{\e_1}>0\}$.  
  If  $x_1 \in \partial\{w_{\e_1}>0\} \cap \{u>0\}$ and $|x_1|\neq 1$, we again obtain a contradiction since
   $w_{\e_1} \geq u$, $|\nabla w_{\e_1}(x_1)|<1$, and $|\nabla u(x_1)|=1$. We now consider the last case in which $|x_1|=1$ and 
   $x_1 \in \partial\{w_{\e_1}>0\} \cap \{u>0\}$. Since $|\nabla v_{\e_1}|<1$ on $\partial B_1 \cap \partial \{v_{\e_1}>0\}$ and 
   $|\nabla \tilde{v}_{\e_1}| <1$ on $\partial B_{1} \cap \{\tilde{v}_{\e_1}>0\}$. Then 
    \[
     \sup_{B_{t}(x_1)} w_{\e_1} \leq \delta_3 t,
    \]
    for $x_1 \in \partial B_1 \cap \partial \{w_{\e_1}>0\}$ and for some constant $0<\delta_3<1$ and $t$ small enough. 
    We then again obtain a contradiction since $w_{\e_1}\geq u$. 
     Therefore, $u \leq \Phi_c$ on $B_{\rho}$. Since
    $\Phi_c$ is homogeneous, then by rescaling it is also true that 
   $u\leq \Phi_c$ for any solution $u$ to \eqref{e:phase1}  with $u = \Phi_c$ on $\partial B_1$.   
  \end{proof}

We now give the proof of our second main Theorem. 

\begin{proof}[Proof of Theorem \ref{t:main2}]
 Let $0\leq c \leq c_0$. From the Calculus of Variations there exists a minimizer $u$ of \eqref{e:conefunct} with $u = \Phi_c$ on $\partial B_1$. 
 By Proposition \ref{p:solution} the minimizer $u$ is a solution to \eqref{e:phase1} in $B_1$.  By  Lemma \ref{l:above1} we have  
 $u \leq \Phi_c$. The $c_0$ in the statement of Lemma \ref{l:below1} is greater than or equal to the $c_0$ in the statement of Lemma \ref{l:above1}, so that from 
 Lemma \ref{l:below1} we also have 
  $u \geq \Phi_c$ in $B_1$. Then $u \equiv \Phi_c$, and therefore $\Phi_c$ is a minimizer. 
 \end{proof}

\appendix

\section{A maximum principle}
 In order to prove nonnegative mean curvature of the free boundary of a homogeneous solution, we will need two Lemmas. If $v=rf(\theta)$ so that $v$
 is homogeneous of degree $1$, then 
  \[
    |\nabla v|^2 = f^2 + |\nabla_\theta f|^2. 
  \]
 Consequently, 
 \[
    v^2(x) \leq |\nabla v(x)|^2 \quad \text{ for any } x \in S^{n-1}. 
 \]
If  $u$ is homogeneous of degree $0$, we have a similar result for the Hessian and gradient. Although the following Lemma is not difficult to 
show in all dimensions via induction, we only  state and prove it for three dimensions. 

\begin{lemma}  \label{l:gradhess}
 Let $u:\mathbb{R}^3 \to \mathbb{R}$ and assume $u$ is homogeneous of degree $0$, then 
  \[
   \sum_{i,j=1}^3 u_{x_i x_j}^2(x) \geq 2|\nabla u(x)|^2  \quad \text{ for any } x \in S^{n-1}. 
  \]
\end{lemma}

\begin{proof}
 If  $\lambda_i$ are the eigenvalues of the Hessian, 
  \[
   \sum_{i,j=1}^3 u_{x_i x_j}^2(x)  = \sum_{i=1}^3 \lambda_i^2,
 \] 
 which is invariant under rotation. Therefore,  we may assume without loss of generality that $x_0 \in S^{n-1}$ is $x_0 = (1,0,0)$, so that under 
 spherical coordinates $\theta=0$ and $\phi = \pi/2$. One may then explicitly compute that at $x_0$ we have 
  \[
   \begin{aligned}
   &u_{x x}(x_0) =0 \qquad &u_{yy}(x_0) = u_{\theta \theta}(x_0)  \qquad &u_{x y}(x_0)= - u_{\theta}(x_0) \\
   &u_{zz}(x_0) = u_{\phi \phi}(x_0) &\qquad u_{xz}(x_0) = u_{\phi}(x_0) \qquad &u_{y z} (x_0) = -u_{\theta \phi}(x_0). 
   \end{aligned}
  \]
  Then at $x_0$ we obtain 
   \[
     \sum_{i,j=1}^3 u_{x_i x_j}^2 = u_{\theta \theta}^2 + u_{\phi \phi}^2 + 2u_{\theta \phi}^2 + 2u_{\theta}^2 + 2u_{\phi}^2. 
   \]
   Since at $x_0=(1,0,0)$ we have
   \[
    |\nabla u(x_0)|^2 = |\nabla_{\theta} f(x_0)|^2 = f_{\theta}^2 + f_{\phi}^2,  
   \]
   we conclude that 
   \[
    \sum_{i,j=1}^3 u_{x_i x_j}^2(x) \geq 2|\nabla u(x)|^2 = 2|\nabla_{\theta} f(x)|^2 \quad \text{ for any } x \in S^{n-1}.    
  \]
\end{proof}

We also have the following
 \begin{lemma}  \label{l:max}
  Let $f:S^2 \to \mathbb{R}$ with $u=r^{\alpha}f$ such that $\Delta u=0$ in $\{u>0\}$. If $0<\alpha \leq 1$, then $|\nabla_{\theta} f|^2$ achieves its maximum on 
  $S^{n-1}\cap \partial \{u>0\}$. 
 \end{lemma}

\begin{proof}
 Since $\Delta u=0$, then $\Delta_{\theta} f = -\alpha(\alpha+1) f$. If $v=f$, then 
  \[
   \begin{aligned}
   \Delta |\nabla v|^2 &= 2 \|D^2 v\|^2 + 2 \langle \nabla v, \nabla \Delta v \rangle \\ 
                                 &= 2 \| D^2 v \| ^2 - 2 \alpha (\alpha +1)  |\nabla v|^2 \\
                                 &\geq 0. 
   \end{aligned}
  \] 
  The last inequality is a result of Lemma \ref{l:gradhess}. Then $|\nabla v|^2$ is subharmonic in $\{u>0\}$, and consequently achieves the maximum on the boundary. 
  Since $|\nabla v|^2 = |\nabla_{\theta} f|^2$, the conclusion of the Lemma is immediate. 
\end{proof}

\section{$C^1$ convergence on a wedge-type domain}

In this appendix we show $C^1$ convergence of $w_{\e} \to w_{0}$ on $\overline{U}$ where $w_{\e},w_0,U$ are defined in the proof of Lemma \ref{l:above1}. 
We recall that 
 $U:=\{x \in B_1: x_3> \cos \phi_2\}$ where $\phi_2>\pi/2$ and was fixed in the proof of Lemma
\ref{l:above1}. It is clear that $w_{\e} \to w_0$ in $C^1$ except at the corner $\{x \in \partial B_1: x_3 = \cos \phi_2\}$.  We handle the issue of the corner 
with a series of Lemmas. 

\begin{lemma}   \label{l:bhp}
Let $V:= \{(x,y,z) \in \R^3: 0 \leq \arctan(y/x) \leq \theta_0 < \pi\}$. If $\Delta u =0$ and $u=0$ on $\partial V$, and $|u| \leq C|x|$ for $|x|\geq 1$ and some $C>0$, then 
$u \equiv 0$.  
\end{lemma}

\begin{proof}
 Let $v(r,\theta, x_3)= r^{\lambda} \sin(\lambda \theta)$ where $\lambda \theta_0 =\pi$. Then $\Delta v=0$ in $V$, $v=0$ on $\partial V$, and $v\geq 0$ in $V$. 
 Let $w$ be the harmonic lifting of $u_+$ on $B_R\cap V$. We will choose $R$ large and use the boundary Harnack principle \cite{jk82}. 
 \[
  \sup_{B_{R/2}\cap V} \frac{u_+}{v} \leq C_1 \sup_{B_{R/2}\cap v} \frac{w}{v} \leq \inf_{B_{R/2}\cap U} \frac{w}{v} \leq C_2 \frac{R}{R^{\lambda}}. 
 \]
 Since $\lambda >1$, as $R \to \infty$ we obtain that $u_+\equiv 0$. The same argument applies to $u_-$. 
\end{proof}

\begin{lemma}  \label{l:lipbound}
 Let $\Delta_c u =0$ in $U$ with $u=0$ on $\{x_3 =\cos \phi_2\}\cap \overline{B}_1$ and $u \in C^1 (\overline{U})$.  Then there exists $C<\infty$ 
 depending on $\displaystyle \| u\|_{C^1(\overline{U})}$ and $c_0$ but independent of $c$ if $0\leq c \leq c_0$ such that if $x_0 \in \partial B_1 \cap \{x_3 = \cos \phi_2\}$, then 
  \[
   |u| \leq C|x-x_0|.
  \]
\end{lemma}

\begin{proof}
 We first translate so that $x_0 =0$ the origin. 
 We will use compactness combined with a blow-up similar to the argument in the proof of Theorem 6.1 in \cite{ALP}. 
  Suppose by way of contradiction that no such $C$ exists.  
  Then there exists $u_j, c_j , r_j$ with 
   $\Delta_{c_j} u_j =0$ and $r_j \to 0$ such that if 
   \[
    S_{r_j} = \sup_{\Omega \cap B_{r_j}} |u_j|
   \]
  then 
  \[
   \begin{aligned}
    (i) &\quad \frac{S_{r_j}}{r_j} \to \infty \\
    (ii) &\quad  S_{r_j2^k} \leq 2^k S_{r_j} \text{ for } k \in \mathbb{N} \text{ with } 2^k r_j \leq 1. 
   \end{aligned}
  \]
  We let
  \[
   u_{r_j} := \frac{u(r_j x)}{S_{r_j}}. 
  \]
  We have the following 
   \[
    \begin{aligned}
    (1) &\quad \tilde{\Delta}_{c_j} u_{r_j} =0 \\
    (2) &\quad \sup_{B_{2^k}} |u_{r_j}| \leq 2^k \text{ whenever } 2^k r_j\leq 1 \\
    (3) &\quad u_{r_j}(0)=0 \\
    (4) &\quad \sup_{B_1} |u_{r_j}|=1. 
    \end{aligned}
   \]
   Then $c_j \to c$, and $u_{r_j} \to u_0$ with 
   \[
    \begin{cases}
        a^{ij} \partial_{ij} u_0 =0 \\
      u_0 =0 \text{ on } \partial U \\
      u_0 \leq |x| \text{ for } |x|\geq 1,
    \end{cases}
   \]
   where after rotation
    \[ 
      a^{ij}=\frac{1}{(1+c^2)}
        \left( 
          \begin{array}{ccc}
          1+c^2                              & 0                                                & 0 \\
          0                                      & 1 + c^2\cos^2 \phi_2              & -c^2 \sin \phi_2 \cos \phi _2 \\
          0                                      & -c^2 \sin \phi_2 \cos \phi_2       &1+c^2 \sin^2 \phi_2 
        \end{array} 
       \right).
\]  
 Since $a^{ij}$ is a constant coefficient matrix, a linear change of variables in only the $x_2$ and $x_3$ variables will give a new 
 solution $\tilde{u}_0$  with $\Delta \tilde{u}_0=0$. For $0\leq c \leq c_0$
 with $c_0$ small, the transformed domain will still be a wedge-domain $V$ with angle less than $\pi$.  
 Then by Lemma \ref{l:bhp}, it follows that $\tilde{u}_0 \equiv 0$ so that $u_0\equiv 0$. 
 This contradicts the fact that $\displaystyle \sup_{U\cap B_1} u_0 =1 $.  

\end{proof}

\begin{lemma}  \label{l:Uconv}
 Let $u_k$ be a sequence of solutions to $\Delta_{c_k} u_k =0$ in $U$ with $u_k =0$ on $\{x_3 = \cos\phi_2\}\cap \overline{B}_1$. Assume $\lim_{k \to \infty} c_k=0$ and  r that $u_k \to v$ uniformly in 
  $\overline U$ and $u_k \to v$ in $C^1(\partial B_1 \cap \{x_3 \geq \phi_2\})$.  Then $u_k \to v$ in $C^1(\overline{U})$. 
\end{lemma}

\begin{proof}
 Suppose by way of contradiction that there exist points  $x_k \in \partial U$ such that $x_k \to x_0$ and 
  \begin{equation}  \label{e:ep}
   \left |\frac{\partial u_k}{\partial \nu}(x_k) - \frac{\partial v}{\partial \nu}(x_k) \right| > \epsilon,
  \end{equation}
  for some fixed $\e$. Let $|x_k-x_0|=r_k$. We rescale by 
   \[
    \tilde{u}_k = \frac{u_k(r_k (x-x_0))}{r_k}.  
   \]
   and
   \[
    v_k := \frac{v(r_k (x-x_0)}{r_k}. 
   \]
  By Lemma \ref{l:lipbound} we have that 
   \[
    |\tilde{u}_k|, |v_k| \leq C 
   \]
  Now $\tilde{u}_k - v_k \to v_0$ uniformly with $\Delta v_0 =0$. Furthermore, since $u_k \to v$ in $C^1(\partial B_1 \cap \{x_3 \geq \phi_2\})$ it follows that 
  $v_0 \equiv 0$ on $\partial V$ where $V$ is the domain obtained in the blowup and satisfies the assumptions of 
  Lemma \ref{l:lipbound}. Then since $v_0$ has linear growth it follows from Lemma \ref{l:bhp} that $v \equiv 0$. Now 
  from the $C^1$ convergence of $u_k - v_k$ away from the wedge of $V$ and \eqref{e:ep} it follows that there exists $x_0 \in \partial V \cap \partial B_1$ such that 
   $|\nabla v_0| > \epsilon$. But this contradicts the fact that $v \equiv 0$. 
\end{proof}

\bibliographystyle{amsplain}
\bibliography{refhighdcone} 

\end{document}